\documentclass[12pt]{article}
\usepackage[centertags]{amsmath}
\usepackage{amsfonts}
\usepackage{amssymb}
\usepackage{latexsym}
\usepackage{amsthm}
\usepackage{newlfont}
\usepackage{graphicx}
\usepackage{listings}
\usepackage{booktabs}
\usepackage{abstract}
\usepackage{enumerate}
\RequirePackage{srcltx}
\date{}

\newlength{\defbaselineskip}
\newcommand{\setlinespacing}[1]%
           {\setlength{\baselineskip}{#1 \defbaselineskip}}

\newcommand{\N}{{\mathbb{N}}}

\newcommand{\actaqed}{\hfill $\actabox$}
{\medskip\noindent \textit{Proof of #1. }}%
{\actaqed \medskip}

\def\cC{{\mathcal C}}

\def \Tr{\mathcal T}

\def \cR{\mathcal R}

\def \cM{\mathcal M}
\def\R{{\mathbb R}}
\def\Z{\mathbb Z}

\def \T{\mathbb T}

\def \<{\langle}
\def\>{\rangle}

\def \e{\varepsilon}

\def\ga{\gamma}

\def\bx{\mathbf x}

\def\bk{\mathbf k}

\def\bs{\mathbf s}

\def\bW{\mathbf W}

\newtheorem{Theorem}{Theorem}[section]
\newtheorem{Lemma}{Lemma}[section]

\newtheorem{Proposition}{Proposition}[section]

\numberwithin{equation}{section}

\newcommand{\be}{\begin{equation}}
\newcommand{\ee}{\end{equation}}

\begin{document}

\title{Sampling discretization of integral norms of the hyperbolic cross polynomials}
\author{V. Temlyakov}

\newcommand{\Addresses}{{
  \bigskip
  \footnotesize




  V.N. Temlyakov, \textsc{University of South Carolina,\\ Steklov Institute of Mathematics,\\ and Lomonosov Moscow State University
  \\
E-mail:} \texttt{temlyak@math.sc.edu}



}}

\maketitle
\begin{abstract}
{The paper is devoted to discretization of integral norms of functions from
a given finite dimensional subspace. We use recent general results on sampling discretization to derive a new Marcinkiewicz type
   discretization theorem for  the multivariate trigonometric polynomials with frequencies from the hyperbolic crosses. It is shown that recently developed techniques allow us to improve the known results in this direction.
  }
\end{abstract}

\section{Introduction}
\label{I}

Let $\Omega$ be a compact subset of $\R^d$ with the probability measure $\mu$. By $L_q$, $1\le q< \infty$, norm we understand
$$
\|f\|_q:=\|f\|_{L_q(\Omega)} := \left(\int_\Omega |f|^qd\mu\right)^{1/q}.
$$
By discretization of the $L_q$ norm we understand a replacement of the measure $\mu$ by
a discrete measure $\mu_m$ with support on a set $\xi =\{\xi^j\}_{j=1}^m \subset \Omega$. This means that integration with respect to measure $\mu$ is replaced by an appropriate cubature formula. Thus, integration is replaced by evaluation of a function $f$ at a
finite set $\xi$ of points. This is why we call this way of discretization {\it sampling discretization}. Discretization is
a very important step in making a continuous problem computationally feasible. The reader can find a corresponding discussion in a recent survey \cite{DPTT}. The first results in sampling discretization were obtained by Marcinkiewicz and
by Marcinkiewicz-Zygmund (see \cite{Z}) for discretization of the $L_q$ norms of the univariate trigonometric polynomials in 1930s. We call discretization results of this kind the {\it Marcinkiewicz type theorems}. Recently, a substantial progress in sampling discretization has been made in \cite{VT158}, \cite{VT159}, \cite{DPTT}, \cite{VT168}, \cite{DPSTT1}, \cite{DPSTT2}, \cite{Kos}. To discretize  the integral norms successfully, a  new technique was introduced.  This technique takes different forms in different papers but the common feature of its forms is the following. The new sampling discretization technique is a combination of probabilistic technique, in particular chaining technique, with results on the entropy numbers in the uniform norm (or its variants). Fundamental results from \cite{BLM}, \cite{Tal}, \cite{MSS} were used. The reader can find 
results on chaining in \cite{KoTe}, \cite{Tbook} and on the generic chaining in \cite{Tal}.  We note that the idea of chaining technique goes back to the 1930s, when it was suggested by A.N. Kolmogorov. Later, these type of results have been developed in the study of the central limit theorem in probability theory (see, for instance, \cite{GZ}).
Also, the reader can find general results on metric entropy in \cite[Ch.15]{LGM},  \cite[Ch.3]{Tbook}, \cite[Ch.7]{VTbookMA}, \cite{Carl},
\cite{Schu} and in the recent papers \cite{VT156} and  \cite{HPV}.
Bounds for the entropy numbers of function classes are important by themselves and also have important connections to other fundamental problems (see, for instance, \cite[Ch.3]{Tbook} and \cite[Ch.6]{DTU}).

We now proceed to the detailed presentation.

{\bf Marcinkiewicz problem.} Let $\Omega$ be a compact subset of $\R^d$ with the probability measure $\mu$. We say that a linear subspace $X_N$ (index $N$ here, usually, stands for the dimension of $X_N$) of $L_q(\Omega)$, $1\le q < \infty$, admits the Marcinkiewicz type discretization theorem with parameters $m\in \N$ and $q$ and positive constants $C_1\le C_2$ if there exist a set 
$$
\Big\{\xi^j \in \Omega: j=1,\dots,m\Big\}
$$ 
 such that for any $f\in X_N$ we have
\be\label{I.1}
C_1\|f\|_q^q \le \frac{1}{m} \sum_{j=1}^m |f(\xi^j)|^q \le C_2\|f\|_q^q.
\ee
In the case $q=\infty$ we define $L_\infty$ as the space of continuous functions on $\Omega$  and ask for
\be\label{I.2}
C_1\|f\|_\infty \le \max_{1\le j\le m} |f(\xi^j)| \le  \|f\|_\infty.
\ee
We will also use the following brief way to express the above properties: the $\cM(m,q)$ (more precisely the $\cM(m,q,C_1,C_2)$) theorem holds for  a subspace $X_N$, written $X_N \in \cM(m,q)$ (more precisely $X_N \in \cM(m,q,C_1,C_2)$).

Our main interest in this paper is to discuss the Marcinkiewicz problem in the case, when 
$X_N$ is a subspace of the trigonometric polynomials with frequencies (harmonics) from
a hyperbolic cross. 
By $Q$ we denote a finite subset of $\Z^d$, and $|Q|$ stands for the number of elements in $Q$. Let
$$
\Tr(Q):= \left\{f: f=\sum_{\bk\in Q}c_\bk e^{i(\bk,\bx)},\  \  c_{\bk}\in\mathbb{C}\right\}.
$$
For $\bs\in\Z^d_+$
define
$$
\rho (\bs) := \{\bk \in \Z^d : [2^{s_j-1}] \le |k_j| < 2^{s_j}, \quad j=1,\dots,d\}
$$
where $[x]$ denotes the integer part of $x$. We define the step hyperbolic cross
$Q_n$ as follows
$$
Q_n := \cup_{\bs:\|\bs\|_1\le n} \rho(\bs)
$$
and the corresponding set of the hyperbolic cross polynomials as 
$$
\Tr(Q_n) := \{f: f=\sum_{\bk\in Q_n} c_\bk e^{i(\bk,\bx)}\}.
$$
In addition to the step hyperbolic cross $Q_n$ we also consider a more general step hyperbolic 
cross $Q_n^\ga$, where $\ga = (\ga_1,\dots,\ga_d)$ has the form $1=\ga_1=\dots=\ga_\nu< \ga_{\nu+1} \le \dots\le \ga_d$ with $\nu\in \N$, $\nu \le d$:
$$
Q_n^\ga := \cup_{\bs:(\ga,\bs)\le n} \rho(\bs).
$$
It is clear that in the case $\ga = {\mathbf 1} := (1,\dots,1)$ we have $Q_n^{\mathbf 1} = Q_n$.
In this paper we are primarily interested in the Marcinkiewicz type discretization theorems for the hyperbolic cross trigonometric polynomials from $\Tr(Q_n^\ga)$.

The most complete results on sampling discretization are obtained in the case $q=2$.  The problem is basically solved in the case of subspaces of trigonometric polynomials $\Tr(Q)$ with arbitrary $Q$.
In \cite{VT158} it was shown how to derive the following result from the
 recent paper by  S.~Nitzan, A.~Olevskii, and A.~Ulanovskii~\cite{NOU}, which in turn is based on the paper of A.~Marcus, D.A.~Spielman, and N.~Srivastava~\cite{MSS}.

\begin{Theorem}\label{IT1} There are three positive absolute constants $C_1$, $C_2$, and $C_3$ with the following properties: For any $d\in \N$ and any $Q\subset \Z^d$   there exists a set of  $m \le C_1|Q| $ points $\xi^j\in \T^d$, $j=1,\dots,m$ such that for any $f\in \Tr(Q)$
we have
$$
C_2\|f\|_2^2 \le \frac{1}{m}\sum_{j=1}^m |f(\xi^j)|^2 \le C_3\|f\|_2^2.
$$
\end{Theorem}
In other words, there exist three positive absolute constants $C_1$, $C_2$, and $C_3$ such that for any $Q\subset \Z^d$ we have $\Tr(Q) \in \cM(m,2,C_2,C_3)$ provided $m\ge C_1|Q|$.
We now restrict ourselves to the case $q\in [1,\infty)$, $q\neq 2$ and $Q=Q_n^\ga$. In this paper we provide some bounds on $m$, which guarantee that $\Tr(Q_n^\ga) \in \cM(m,q)$.
The following Theorem \ref{IT2} is the main result of the paper.

\begin{Theorem}\label{IT2} For $q\in [1,\infty)$, and $\ga$, where $\ga = (\ga_1,\dots,\ga_d)$ has the form $1=\ga_1=\dots=\ga_\nu< \ga_{\nu+1} \le \dots\le \ga_d$ with $\nu\in \N$, $\nu \le d$, there are three positive constants
$C_i=C_i(q,\ga)$, $i=1,2,3$, such that we have 
$$
\Tr(Q_n^\ga) \in \cM(m,q,C_2,C_3)\quad \text{provided}\quad m\ge C_1|Q_n^\ga|n^{w(\nu,q) },
$$
 where
$$
w(\nu,1) = 3;\quad w(\nu,q) = 2, \quad q\in (1,2]; 
$$
$$
w(\nu,q) = (\nu-1)(q-2) + \min(q,3),\quad q> 2.
$$

\end{Theorem}

We point out that in case $\nu=1$ the exponent of the extra factor in the bound for $m$  does not grow with $q$: $w(1,q)\le 3$. 
Theorem \ref{IT2} improves the corresponding result of E.S. Belinsky \cite{Bel}. In \cite{Bel} 
there is the following condition on the number of sampling points $m\ge C_1|Q_n^\ga|n^{\max((d-1)(q-2),0)+4}$. Theorem \ref{IT2} improves the bound from \cite{Bel}  for all $q\in [1,\infty)$. 
We note that in the case $q=2$ Theorem \ref{IT1} provides a stronger than Theorem \ref{IT2} result. In the case $q\in [1,3]$ Theorem \ref{IT2} follows from
known general results. A new bound proved in this paper corresponds to the case $q\in (3,\infty)$. We present this proof in Section \ref{A}. In Section \ref{B} we discuss entropy numbers of the $L_q$ unit balls of $\Tr(Q_n^\ga)$ in the uniform norm. Finally, in Section \ref{D} we discuss extension of Theorem \ref{IT2} to the case of arbitrary $Q\subset \Z^d$.

Theorem \ref{IT2} does not cover the case $q=\infty$. It is known that the sampling discretization results in case $q=\infty$ are fundamentally different from those in case $q\in [1,\infty)$. Theorem \ref{IT2} shows that for all $q\in [1,\infty)$ condition $m\ge C(q,d)|Q_n|n^{w(d,q)}$ is sufficient for $\Tr(Q_n) \in \cM(m,q)$. An extra factor $n^{w(d,q)}$ is a logarithmic 
in terms of $|Q_n|$ factor. A nontrivial surprising negative result was proved for $q=\infty$ (see \cite{KT3}, \cite{KT4}, \cite{KaTe03}, and, also, \cite{VTbookMA}, p.344, Theorem 7.5.17). The authors proved that the necessary condition for
$\Tr(Q_n)\in\cM(m,\infty)$ is $m\ge C|Q_n|^{1+c}$ with absolute constants $C,c>0$. We do not present new results for the case $q=\infty$ in this paper. The reader can find further results and 
discussions of the case $q=\infty$ in \cite{DPTT} and \cite{VT168}.

There are many open problems in sampling discretization (see \cite{DPTT}). We now formulate
one directly related to Theorems \ref{IT1} and \ref{IT2}.

{\bf Open problem 1.}  Is it true that for $q\in [1,\infty)$, $q\neq 2$, and $\ga$ there are three positive constants  
$C_i=C_i(q,\ga)$, $i=1,2,3$, such that we have 
$$
\Tr(Q_n^\ga) \in \cM(m,q,C_2,C_3)\quad \text{provided}\quad m\ge C_1|Q_n^\ga|.
$$

Throughout the paper letter $C$ denotes a positive constant, which may be different in different formulas. Notation $C(q,d)$ means that the constant $C$ may depend on parameters $q$ and $d$. 
Sometimes it will be convenient for us to use the following notation. For two sequences $\{a_k\}_{k=1}^\infty$ and $\{b_k\}_{k=1}^\infty$ we write $a_k \asymp b_k$ if there are two positive constants $C_1$ and $C_2$ independent of $k$ such that $C_1a_k \le b_k\le C_2 a_k$, $k=1,2,\dots$. 

\section{General results and proof of Theorem \ref{IT2}}
\label{A}

We begin with a simple remark on a connection between real and complex cases. Usually, 
general results are proved for real subspaces $X_N$. Suppose that a complex subspace 
has a form
$$
\cC_N = \{f = f_R + if_I,\quad f_R, f_I \in X_N\},
$$
where $X_N$ is a real subspace. 

\begin{Proposition}\label{AP1} Let $q\in [1,\infty)$. Suppose $X_N\in \cM(m,q,C_2,C_3)$. 
Then $\cC_N \in \cM(m,q,C_22^{-q-1},C_32^{q+1})$.

\end{Proposition}
\begin{proof} Using a simple inequality for a complex number $z=x+iy$
$$
\max(|x|,|y|) \le |z| \le |x|+|y|
$$
we obtain the following inequalities. Let $\xi=\{\xi^j\}_{j=1}^m$ be such that for all $g\in X_N$
\be\label{A1}
C_2 \|g\|_q^q  \le \frac{1}{m} \sum_{j=1}^m |g(\xi^j)|^q \le C_3\|g\|_q^q.
\ee
Denote
$$
g_\xi := (g(\xi^1),\dots,g(\xi^m)),\quad f_{R,\xi} := (f_R)_\xi,\quad f_{I,\xi} := (f_I)_\xi
$$
and
$$
\|g_\xi\|_{\ell_{q,m}}^q :=  \frac{1}{m} \sum_{j=1}^m |g(\xi^j)|^q .
$$
Then for $f\in \cC_N$ we obtain
$$
\|f\|_q^q \le 2^q(\|f_R\|_q^q + \|f_I\|_q^q) \le 2^qC_2^{-1}(\|f_{R,\xi}\|_{\ell_{q,m}}^q + \|f_{I,\xi}\|_{\ell_{q,m}}^q) \le 2^{q+1}C_2^{-1}\|f_{\xi}\|_{\ell_{q,m}}^q.
$$
and
$$
\|f_{\xi}\|_{\ell_{q,m}}^q \le 2^q(\|f_{R,\xi}\|_{\ell_{q,m}}^q + \|f_{I,\xi}\|_{\ell_{q,m}}^q) \le 
2^qC_3(\|f_R\|_q^q +\|f_I\|_q^q) \le 2^{q+1}C_3\|f\|_q^q.
$$
This proves Proposition \ref{AP1}.
\end{proof}

Thus, it is sufficient to prove Theorem \ref{IT2} for the subspace $\cR\Tr(Q_n^\ga)$ of real trigonometric polynomials from $\Tr(Q_n^\ga)$. Our proof is based on conditional theorems. 
We now formulate the known conditional theorems.   

We begin with the definition of the entropy numbers.
  Let $X$ be a Banach space and let $B_X$ denote the unit ball of $X$ with the center at $0$. Denote by $B_X(y,r)$ a ball with center $y$ and radius $r$: $\{x\in X:\|x-y\|\le r\}$. For a compact set $A$ and a positive number $\e$ we define the covering number $N_\e(A)$
 as follows
$$
N_\e(A) := N_\e(A,X) 
:=\min \{n : \exists y^1,\dots,y^n, y^j\in A :A\subseteq \cup_{j=1}^n B_X(y^j,\e)\}.
$$
It is convenient to consider along with the entropy $H_\e(A,X):= \log_2 N_\e(A,X)$ the entropy numbers $\e_k(A,X)$:
$$
\e_k(A,X)  :=\inf \{\e : \exists y^1,\dots ,y^{2^k} \in A : A \subseteq \cup_{j=1}
^{2^k} B_X(y^j,\e)\}.
$$
In our definition of $N_\e(A)$ and $\e_k(A,X)$ we require $y^j\in A$. In a standard definition of $N_\e(A)$ and $\e_k(A,X)$ this restriction is not imposed. 
However, it is well known (see \cite{Tbook}, p.208) that these characteristics may differ at most by a factor $2$. Throughout the paper we use the following notation for the unit $L_q$ ball of $X_N$
$$
X_N^q:= \{f\in X_N:\, \|f\|_q \le 1\}.
$$

The first conditional theorem in the sampling discretization was proved in \cite{VT159} in the case $q=1$.
\begin{Theorem}\label{AT1} Suppose that a subspace $X_N$ satisfies the condition $(B\ge 1)$
$$
\e_k(X^1_N,L_\infty) \le  B\left\{\begin{array}{ll}  N/k, &\quad k\le N,\\
 2^{-k/N},&\quad k\ge N.\end{array} \right.
$$
Then for large enough absolute constant $C$ there exists a set of  $$m \le CNB(\log_2(2N\log_2(8B)))^2$$ points $\xi^j\in \Omega$, $j=1,\dots,m$,   such that for any $f\in X_N$ 
we have
$$
\frac{1}{2}\|f\|_1 \le \frac{1}{m}\sum_{j=1}^m |f(\xi^j)| \le \frac{3}{2}\|f\|_1.
$$
\end{Theorem}

Proof of Theorem \ref{AT1} in \cite{VT159} is based on the concentration measure result from 
\cite{BLM} (see Lemma 2.1 in \cite{VT159}) and on the elementary chaining type technique 
from \cite{KoTe} (see also \cite{Tbook}, Ch.4). Theorem \ref{AT1} was extended to the case 
$q\in [1,\infty)$ in \cite{DPSTT1}. 

\begin{Theorem}\label{AT2} Let $1\le q<\infty$. Suppose that a subspace $X_N$ satisfies the condition
\be\label{A2}
\e_k(X^q_N,L_\infty) \le  B (N/k)^{1/q}, \quad 1\leq k\le N,
\ee
where $B\ge 1$.
Then for large enough constant $C(q)$ there exists a set of
$$
m \le C(q)NB^{q}(\log_2(2BN))^2
$$
 points $\xi^j\in \Omega$, $j=1,\dots,m$,   such that for any $f\in X_N$
we have
$$
\frac{1}{2}\|f\|_q^q \le \frac{1}{m}\sum_{j=1}^m |f(\xi^j)|^q \le \frac{3}{2}\|f\|_q^q.
$$
\end{Theorem}

Note that it is well known that the inequality 
$$
\e_k(X^q_N,L_\infty) \le  B (N/k)^{1/q}, \quad 1\leq k\le N
$$
implies 
$$
\e_k(X^q_N,L_\infty) \le  6B 2^{-k/N}, \quad   k > N.
$$
In the same way as in the proof of Theorem \ref{AT1} proof of Theorem \ref{AT2} in  \cite{DPSTT1} is based on the concentration measure result from 
\cite{BLM}. However, the chaining technique, used in  \cite{DPSTT1} differs from that in 
\cite{VT159}. In  \cite{DPSTT1} the corresponding $\e$-nets are built in a more delicate way 
than in \cite{VT159} (sandwiching technique). In both Theorems \ref{AT1} and \ref{AT2} the conditions are formulated in terms of the entropy numbers in the uniform norm $L_\infty$. Very recently, a new idea in this direction was developed in \cite{Kos}. E. Kosov in \cite{Kos} proves the corresponding   theorem with the conditions imposed on the entropy numbers in a weaker metric than the uniform norm. We now formulate his result. 

Let $Y_s:= \{y_j\}_{j=1}^s \subset \Omega$ be a set of sample points from the domain $\Omega$. Introduce a semi-norm 
$$
\|f\|_{Y_s}:=\|f\|_{L_{\infty}(Y_s)} := \max_{1\le j\le s} |f(y_j)|.
$$
Clearly, for any $Y_s$ we have $\|f\|_{Y_s} \le \|f\|_\infty$. 
The following result is from \cite{Kos} (see Corollary 3.4 there). 

\begin{Theorem}\label{AT3} Let $1\le q<\infty$. There exists a number $C_1(q)>0$ such that for $m$ and $B$ satisfying 
$$
m\ge C_1(q) NB^q (\log N)^{w(q)},\quad w(1):=2, \quad w(q) := \max(q,2)-1, 1<q<\infty,
$$
and for a subspace $X_N$ satisfying the condition: for any set $Y_m\subset \Omega$
\be\label{A3}
\e_k(X^q_N,L_{\infty}(Y_m)) \le  B (N/k)^{1/q}, \quad 1\leq k\le N
\ee
 there are
 points $\xi^j\in \Omega$, $j=1,\dots,m$,   such that for any $f\in X_N$
we have
$$
\frac{1}{2}\|f\|_q^q \le \frac{1}{m}\sum_{j=1}^m |f(\xi^j)|^q \le \frac{3}{2}\|f\|_q^q.
$$
\end{Theorem}

{\bf Proof of Theorem \ref{IT2}.} We now proceed to the proof of Theorem \ref{IT2}. We treat separately four cases. We have 
\be\label{A3}
X_N = \cR\Tr(Q_n^\ga),\qquad  N=|Q_n^\ga| \asymp 2^n n^{\nu-1}.
\ee

{\bf 1. Case $q=1$.} We use Theorem \ref{AT1} here. We obtain the required bounds on the entropy numbers from Proposition \ref{BP1}.  It gives us $B=C(\ga)n$. Therefore, by Theorem \ref{AT1} we obtain that
\be\label{A4}
\cR\Tr(Q_n^\ga) \in \cM(m,1), \quad \text{provided}\quad m \ge C(\ga)|Q_n^\ga| n^3,
\ee
which is claimed in Theorem \ref{IT2}.

{\bf 2. Case $q\in (1,2]$.} We use Theorem \ref{AT3} here. We obtain the required bounds on the entropy numbers from Proposition \ref{BP1}.  It gives us $B=C(q,\ga)n^{1/q}$. Therefore, by Theorem \ref{AT3} we obtain that
\be\label{A5}
\cR\Tr(Q_n^\ga) \in \cM(m,q), \quad \text{provided}\quad m \ge C(q,\ga)|Q_n^\ga| n^2,
\ee
which is claimed in Theorem \ref{IT2}.

{\bf 3. Case $q\in (2,3)$.} We use Theorem \ref{AT3} here. We obtain the required bounds on the entropy numbers from Lemma \ref{BL2}. It is well known that the condition $\Tr(Q_n^\ga) \in \cM(s,\infty)$ with $s\le C(d)2^{nd}$ holds (see the argument in the proof of Proposition \ref{BP1}). It gives us $B=C(q,d)n^{1/q}(MN^{-1/q})$ where $M$ is from the Nikol'skii inequality (\ref{B8}). It is known (see \cite{Tem4} and \cite{Tmon})
that 
$$
M\asymp 2^{n/q} n^{(\nu-1)(1-1/q)}.
$$
Therefore,
$$
MN^{-1/q} \asymp n^{(\nu-1)(1-2/q)}, \quad B^q\asymp n^{(\nu-1)(q-2)+1}.
$$
By Theorem \ref{AT3} we obtain that
\be\label{A6}
\cR\Tr(Q_n^\ga) \in \cM(m,q), \quad \text{provided}\quad m \ge C(q,\ga)|Q_n^\ga| n^{(\nu-1)(q-2) + q},
\ee
which is claimed in Theorem \ref{IT2}.

{\bf 4. Case $q\in [3,\infty)$.} We use Theorem \ref{AT2} here. In the same way as above in Case 3 we get
$$
B^q\asymp n^{(\nu-1)(q-2)+1}.
$$
By Theorem \ref{AT2} we obtain that
\be\label{A7}
\cR\Tr(Q_n^\ga) \in \cM(m,q), \quad \text{provided}\quad m \ge C(q,\ga)|Q_n^\ga| n^{(\nu-1)(q-2) + 3},
\ee
which is claimed in Theorem \ref{IT2}.

\section{Bounds of the entropy numbers}
\label{B}

In this section we obtain bounds of the entropy numbers $\e_k(\cR\Tr(Q_n^\ga)^q,L_\infty)$ in the form (\ref{A2}). Some results on the entropy numbers $\e_k(\cR\Tr(Q_n)^q,L_\infty)$ can be found in \cite{VTbookMA}, Ch. 7. However, those results are designed for proving upper bounds of the entropy numbers of classes of functions with mixed smoothness. Here 
we obtain bounds, which serve better for the sampling discretization (see a detailed discussion of such a comparison in \cite{VT159}, Section 7). We begin with the case $q\in [1,2]$. 
The following result is from \cite{DPSTT2}.

\begin{Theorem}\label{BT1}	Assume that  $X_N$ is  an $N$-dimensional subspace of $L_\infty(\Omega)$ satisfying the following two conditions:
	\begin{enumerate}
		\item [{\bf \textup{(i)} }] There exists a constant $K_1>1$ such that
		\begin{equation}\label{B1}
		\|f\|_\infty \leq (K_1 N)^{1/2}\|f\|_2,\   \ \forall f\in X_N.
		\end{equation}

		\item [{\bf \textup{(ii)}}] There exists a constant $K_2>1$  such that
		\begin{equation}\label{B2}
		\|f\|_\infty \leq K_2 \|f\|_{\log N},\   \ \forall f\in X_N.
		\end{equation}
	\end{enumerate}
Then for  each $1\leq q\leq 2$,  there exists a constant $C(q)>0$ depending only on $q$ such that
	\begin{equation}\label{B3}
	\e_k (X_N^q, L_\infty) \leq C(q) (K_1K_2^2\log N)^{1/q} \begin{cases}
	\bigl(N/k\bigr)^{1/q},&\  \ \text{if $1\leq k\leq N$},\\
	2^{-k/N}, &\   \ \text{if $k>N$}.
	\end{cases}
	\end{equation}
 
\end{Theorem}

We now apply Theorem \ref{BT1} to the case $X_N=\cR\Tr(Q_n^\ga)$. Clearly, in this case 
$N=\dim \cR\Tr(Q_n^\ga) = |Q_n^\ga| \asymp 2^n n^{\nu-1}$. 

\begin{Proposition}\label{BP1} Let $q\in [1,2]$. We have the bound
\be\label{B4}
\e_k(\cR\Tr(Q_n^\ga)^q,L_\infty) \le C(q,\ga) n^{1/q} (|Q_n^\ga|/k)^{1/q}.
\ee
\end{Proposition}
\begin{proof} It is well known and easy to check that condition (i) is satisfied with $K_1=1$. 
Condition (ii) follows from known results on sampling discretization. We now explain it.
Let $\Pi_n := [-2^n,2^n]^d$ be a $d$ dimensional cube. It is known (see, for instance, \cite{VTbookMA}, p.102, Theorem 3.3.15) that there exists a set $Y_s$ with $s \le C_1(d)2^{nd}$ such that for all $1\le p\le \infty$ and $f\in \cR\Tr(\Pi_n)$
\be\label{B5}
C_2(d)\|f\|_p \le \|f_{Y_s}\|_{\ell_{p,s}} \le C_3(d)\|f\|_p.
\ee
It remains to note that $Q_n^\ga \subset \Pi_n$ and that in $\R^s$ we have 
\be\label{B6}
\|\bx\|_{\ell_\infty} \le C(a) \|\bx\|_{\ell_{a\log s,s}}.
\ee
\end{proof}

We now proceed to the case $q\in (2,\infty)$. We will use the following result from \cite{Kos},
which was proved with a help of deep results from functional analysis (see \cite{Tal}, p.552, Lemma 16.5.4). 

\begin{Lemma}\label{BL1} Let $q\in (2,\infty)$. Assume that for any $f\in X_N$ we have
\be\label{B6'}
\|f\|_\infty \le M\|f\|_q
\ee
with some constant $M$. Then for $k\in [1,N]$ we have for any $Y_s$
\be\label{B7}
\e_k(X_N^q, L_\infty(Y_s)) \le C(q)(\log s)^{1/q} (MN^{-1/q})(N/k)^{1/q} .
\ee
\end{Lemma}

We would like to estimate the entropy numbers in the uniform norm. For that purpose we derive 
from Lemma \ref{BL1} the following statement.

\begin{Lemma}\label{BL2} Let $q\in (2,\infty)$. Assume that for any $f\in X_N$ we have
\be\label{B8}
\|f\|_\infty \le M\|f\|_q
\ee
with some constant $M$. Also, assume that $X_N \in \cM(s,\infty)$ with $s\le aN^c$.
Then for $k\in [1,N]$ we have  
\be\label{B9}
\e_k(X_N^q, L_\infty) \le C(q,a,c)(\log N)^{1/q} (MN^{-1/q})(N/k)^{1/q} .
\ee
\end{Lemma}
\begin{proof} Condition $X_N \in \cM(s,\infty)$ means that there exists a set $Y_s$ such that for any $f\in X_N$ we have
\be\label{B10}
\|f\|_\infty \le C_1\|f\|_{Y_s}.
\ee
Lemma \ref{BL1}, relation (\ref{B10}), and inequality $\log s \le c\log N$ imply Lemma \ref{BL2}.

\end{proof}

We now show how to derive a bound on the entropy numbers $\e_k (X_N^q, L_\infty(Y_s))$, obtained in \cite{Kos},  from Theorem \ref{BT1}.

\begin{Proposition}\label{BP2}	Assume that  $X_N$ is  an $N$-dimensional subspace of $L_\infty(\Omega)$ satisfying the following condition:
	\begin{enumerate}
		\item [{\bf \textup{(i)} }] There exists a constant $K_1>1$ such that
		\begin{equation}\label{B11}
		\|f\|_\infty \leq (K_1 N)^{1/2}\|f\|_2,\   \ \forall f\in X_N.
		\end{equation}
	 
	\end{enumerate}
Then for  each $1\leq q\leq 2$ and any $Y_s$ with $\log s \le a \log N$ there exists a constant $C(q,a)>0$ depending only on $q$ and $a$ such that
	\begin{equation}\label{B12}
	\e_k (X_N^q, L_\infty(Y_s)) \leq C(q,a) (K_1\log N)^{1/q} \begin{cases}
	\bigl(N/k\bigr)^{1/q},&\  \ \text{if $1\leq k\leq N$},\\
	2^{-k/N}, &\   \ \text{if $k>N$}.
	\end{cases}
	\end{equation}
 
\end{Proposition}

\begin{proof} Using Lemma 4.3 from \cite{DPSTT2} we discretize simultaneously  the $L_q$ and $L_2$ norms -- replace $\Omega$ by $\Omega_N$ with $8K_1 N^{a+2}\le |\Omega_N| \le CK_1 N^{a+2}$ to get for $p=q$ and $p=2$
$$
\frac{1}{2} \|f\|_p \le \|f\|_{L_p(\Omega_N)} \le \frac{3}{2} \|f\|_p,\quad \|f\|_{L_p(\Omega_N)}^p := \frac{1}{N}\sum_{\omega \in \Omega_N} |f(\omega)|^p.
$$
Note that the Nikol'skii inequality 
(\ref{B11}) implies the inequality
\be\label{B13}
\|f\|_\infty \leq (K_1 N)^{1/q}\|f\|_q,\   \ \forall f\in X_N.
\ee
For a given $Y_s$ consider a new domain $\Omega_S := \Omega_N \cup Y_s$, $|\Omega_S|=S$. Then
$$
\frac{1}{S} \sum_{j=1}^s |f(y_j)|^q \le \frac{1}{S}( s K_1N \|f\|_q^q) \le (8N)^{-1} \|f\|_q^q.
$$
This implies that $\|f\|_{L_q(\Omega)} \asymp \|f\|_{L_q(\Omega_S)}$. In the same way we 
obtain $\|f\|_{L_2(\Omega)} \asymp \|f\|_{L_2(\Omega_S)}$. We now want to apply Theorem \ref{BT1} to $X_N$ restricted to $\Omega_S$. Condition (i) is satisfied because of $\|f\|_{L_2(\Omega)} \asymp \|f\|_{L_2(\Omega_S)}$.
Relation (\ref{B6}) guarantees that condition (ii) of Theorem \ref{BT1} is satisfied in the case of $\Omega_S$. Therefore, applying Theorem \ref{BT1} to $X_N$ restricted to $\Omega_S$ 
we obtain the bounds of the entropy numbers in the metric $L_\infty(\Omega_S)$. Obviously, $\|\cdot\|_{L_\infty(Y_s)} \le \|\cdot\|_{L_\infty(\Omega_S)}$. This completes the proof of Proposition \ref{BP2}. 

\end{proof}

 {\bf A comment on limitations.} In Proposition \ref{BP1}, which is a corollary of Theorem \ref{BT1}, we proved the following bound for $X_N =\cR\Tr(Q_n^\ga)$, $1\le q\le 2$,
 \be\label{B14} 
 \e_k(X^q_N,L_\infty) \le C(q,\ga)(\log N)^{1/q} (N/k)^{1/q}, \quad 1\le k \le N.
 \ee
 
 {\bf Open problem 2.} Could we replace in the bound (\ref{B3}) of Theorem \ref{BT1} $(\log N)^{1/q}$ by the $(\log N)^{\alpha}$ with $\alpha < 1/q$?
 
 It follows from known results on the behavior of the entropy numbers of the classes $\bW^{a,b}_q$ of functions with mixed smoothness that for $1\le q< \infty$ it must be $\alpha \ge 1/2$. We now give a definition of these classes.
 
  Define for $f\in L_1$
$$
\delta_\bs(f) := \sum_{\bk\in\rho(\bs)} \hat f(\bk) e^{i(\bk,\bx)}, \quad \hat f(\bk):= (2\pi)^{-d}\int_{[0,2\pi]^d} f(\bx)e^{-i(\bk,\bx)} d\bx,
$$
and
$$
f_l:=\sum_{\|\bs\|_1=l}\delta_\bs(f), \quad l\in \N_0,\quad \N_0:=\N\cup \{0\}.
$$
  Consider the class (see \cite{VT152})
$$
\bW^{a,b}_q:=\{f: \|f_l\|_q \le 2^{-al}(\bar l)^{(d-1)b}\},\quad \bar l:=\max(l,1).
$$

 Let $X_N := \Tr(Q_n)$ in dimension $d=2$. Then $N\asymp 2^nn$. If (\ref{B14}) holds for 
 all $n\in\N$ then by Theorem 7.7.15 from \cite{VTbookMA}, p.371, we obtain for $a>1/q$
 \be\label{B15}
 \e_k(\bW^{a,b}_q,L_\infty) \le C(q,a,b) k^{-a} (\log k)^{a+b +\alpha}.
 \ee
 On the other hand by Theorem 7.7.10 from  \cite{VTbookMA}, p.365, we obtain for $q=1$, $a>1$
  \be\label{B16}
 \e_k(\bW^{a,b}_1,L_\infty) \asymp k^{-a} (\log k)^{a+b +1/2}.
 \ee
 Also, by Theorem 7.7.14 from  \cite{VTbookMA}, p.365, we obtain for $1<q<\infty$, $a>\max(1/q,1/2)$
  \be\label{B17}
 \e_k(\bW^{a,b}_q,L_\infty) \asymp k^{-a} (\log k)^{a+b +1/2}.
 \ee
Comparing (\ref{B15}) with (\ref{B16}) and (\ref{B17}) we obtain the above claim.

\section{Discussion}
\label{D}

As we pointed out in the Introduction, our main interest in this paper is sampling discretization of the $L_q$ norms, $1\le q<\infty$, of hyperbolic cross polynomials from $\Tr(Q_n^\ga)$. 
Theorem \ref{IT2} provides such results. In this section we discuss the following question: "Is it possible to extend Theorem \ref{IT2} to $\Tr(Q)$ with arbitrary $Q\in \Z^d$?" This discussion will illustrate advantages and limitations of the techniques used in the proof of Theorem \ref{IT2}. Theorem \ref{IT1} gives the sampling discretization result for all $\Tr(Q)$ in case $q=2$. It turns out that 
in case $q\in [1,2)$ Theorem \ref{IT2} can be extended to the case of arbitrary $Q$. 

\begin{Theorem}\label{DT1} For $q\in [1,2)$  there are three positive constants\newline 
$C_i=C_i(q)$, $i=1,2,3$, such that we have for any $Q\in \Z^d$
$$
\Tr(Q) \in \cM(m,q,C_2,C_3)\quad \text{provided}\quad m\ge C_1|Q|(\log (2|Q|))^{w(q) },
$$
 where
$$
w(1) = 3;\quad w(q) = 2, \quad q\in (1,2).
$$
\end{Theorem}
\begin{proof} Let $X_N=\cR\Tr(Q)$, $N=|Q|$. First, we use Proposition \ref{BP2}. Condition (i) of that proposition is satisfied with $K_1=1$. Therefore, Proposition \ref{BP2} guarantees that for  each $1\leq q\leq 2$ and any $Y_s$ with $\log s \le a \log N$ there exists a constant $C(q,a)>0$ depending only on $q$ and $a$ such that
\begin{equation}\label{D1}
\e_k (X_N^q, L_\infty(Y_s)) \leq C(q,a) (\log N)^{1/q} 
	\bigl(N/k\bigr)^{1/q} \quad \text{if} \quad 1\leq k\leq N.
\end{equation}
We now apply Theorem \ref{AT3}. Set $m=s$. Parameter $a$ satisfying $\log m \le a\log N$ will be chosen later. Then by (\ref{D1}) we find $B=C(q,a) (\log N)^{1/q}$. We need to satisfy the following inequality in order to apply Theorem \ref{AT3}
\be\label{D2}
m\ge C_1(q)N C(q,a)^q (\log N)^{1+w(q)}.
\ee
Clearly, for any fixed $a>1$ we can satisfy simultaneously (\ref{D2})) and $\log m \le a\log N$
provided $N\ge C'(q,a)$. Thus, we apply Theorem \ref{AT3} and complete the proof of Theorem \ref{DT1}. 
\end{proof}

We note that a version of Theorem \ref{DT1} with $w(q) =3$, $q\in [1,2]$, follows from Theorem 
\ref{DT2}, which was obtained in \cite{DPSTT2}.

 \begin{Theorem}\label{DT2} Let $X_N$ be an $N$-dimensional subspace of $L_\infty(\Omega)$ satisfying the following condition
$$
\|f\|_\infty\le (K_1N)^{1/2}\|f\|_2, \quad \forall f\in X_N,\quad \log K_1 \le \alpha \log N.
$$
Then for $q\in [1,2]$ we have
$$
X_N \in \cM(m,q)\quad \text{provided}\quad m\ge C(q,\alpha)N(\log N)^3.
$$ 
\end{Theorem}

 We note that the key fact, which allowed us to prove Theorem \ref{DT1}, is the fact that both in Theorem  \ref{DT2} and in Proposition \ref{BP2} we only need the Nikol'skii type inequality (\ref{B11}) between $\|\cdot\|_\infty$ and $\|\cdot\|_2$. This inequality is the same for all $\Tr(Q)$ with $|Q|=N$. We do not know if Theorem \ref{IT2} can be extended to $\Tr(Q)$ with arbitrary 
 $Q\in \Z^d$ in the case $q\in (2,\infty)$. Our proof of Theorem \ref{IT2} in case $q\in (2,\infty)$ is based on the Nikol'skii type inequality (\ref{B8}) between $\|\cdot\|_\infty$ and $\|\cdot\|_q$, which depends on $Q$, not only on $|Q|$. 
 
 {\bf Open problem 3.} Is the following statement true? For $q\in (2,\infty)$, $d\in \N$, there are positive constants $C_i=C_i(q,d)$, $i=1,2,3$, and $c(q,d)$ such that for any $Q\in \Z^d$ we have
 $$
 \Tr(Q) \in \cM(m,q,C_2,C_3) \quad \text{provided} \quad m\ge C_1|Q| (\log (2|Q|))^{c(q,d)}.
 $$
 
 We refer the reader to a list of open problems on sampling discretization in \cite{DPTT}. 

The cornerstone of the above discussed technique of proving the sampling discretization results is the entropy bounds of the type
\be\label{D3}
\e_k(X^q_N,L_\infty) \le  B (N/k)^{1/q}, \quad 1\leq k\le N,\quad 1\le q<\infty.
\ee
Thus, the problem of finding conditions on $X_N$, which guarantee relation (\ref{D3}) is a natural problem. It is well known (see, for instance, \cite{DPSTT1}) that relation (\ref{D3}) for $k=1$  implies the following Nikol'skii type inequality for $X_N$:
\be\label{D4}
	\|f\|_\infty \le 4BN^{1/q}\|f\|_q\   \ \text{for any $f\in X_N$.}
\ee
Therefore, the Nikol'skii type inequality (\ref{D4}) is a necessary condition for (\ref{D3}) to hold.
Lemma \ref{BL2} shows that in the case $q\in (2,\infty)$ condition (\ref{D4}) combined with one more condition $X_N\in \cM(s,\infty)$, $s\le aN^c$, imply a little weaker inequality than in (\ref{D3}): instead of $B$ we get $B'= BC(q,a,c)(\log N)^{1/q}$. 

  {\bf Acknowledgment.} The work was supported by the Russian Federation Government Grant N{\textsuperscript{\underline{o}}}14.W03.31.0031.

 \Addresses

\end{document}